\documentclass{amsart}
\usepackage{graphicx}
\newtheorem{theorem}{Theorem}[section]

\newtheorem{proposition}[theorem]{Proposition}
\theoremstyle{definition}
\newtheorem{definition}[theorem]{Definition}
\theoremstyle{remark}

\numberwithin{equation}{section}

\begin{document}
\newcommand{\CC}{{\mbox{\rm $\scriptscriptstyle ^\mid$\hspace{-0.40em}C}}}
\newcommand{\OO}{{\mathcal O}}
\title[ON ADMISSIBLE LIMITS OF FUNCTIONS
]{ON ADMISSIBLE LIMITS OF  HOLOMORPHIC FUNCTIONS OF SEVERAL COMPLEX VARIABLES
}%
\author{P.V.Dovbush}%
\address{Institute of Mathematics and Computer Science  of  Academy of Sciences  of Moldova, 5 Academy  Street, MD-2028, Kishinev, Republic
of Moldova}%
\email{peter.dovbush@gmail.com}%


\begin{abstract}\textbf{Abstract}
 The aim of the present article is to establish the connection between the existence of
the limit along the normal and an admissible limit at a fixed boundary point for holomorphic
functions of several complex variables.
\end{abstract}
\maketitle
\section{Introduction}

The connection between the existence
of a radial limit  and an angular limit for a holomorphic function defined on the unit disc is described by Lehto and Virtanen  \cite[Theorem 5]{[LV]} in terms of the
growth of the spherical derivative.

For a precise description we introduce several terms and notation.

Let $U=\{z\in \CC: |z|<1\}$ be a unit disc in $\CC.$ Let $\alpha>1.$ A non-tangential region $\Gamma_\alpha(\xi)$ for $\alpha>1$
 and an angular region $A_\theta(\xi)$ for $\theta \in (0, 2\pi)$  at $\xi \in \partial U$ are defined as follows:
$$\Gamma_\alpha(\xi)= \{z \in U: |1-z\overline{\xi}|<\frac{\alpha}{2}(1-|z|^2)\}, \atop
A_\theta(\xi)=\{z \in U: \pi-\theta<\arg(z-\xi)<\pi+\theta\}.$$
It is to be noted that non-tangential regions and angular regions are
equivalent: For every $\alpha>1$ there is a $\theta\in (0,\frac{\pi}{2})$
 such that  $\Gamma_\alpha(\xi)\subset A_\theta(\xi)$
and
for every $\theta\in (0,\frac{\pi}{2})$
 there is an $\alpha>1$ and a disk $d$ centered at $\xi$ such that
$ A_\theta(\xi)\cap d \subset\Gamma_\alpha(\xi).$

To see this let $d_1$ be the be the unit disk with center $\xi,$  $z\in U$ and $\varphi=\pi-\arg(z-\xi).$ From the law of cosines
$$|z|^2=1-2\cos\varphi|\xi-z|+|\xi-z|^2.$$ Since $|\xi|=1$ we have $|\xi-z|=|1-z\overline{\xi}|$ and
$$\frac{|1-z\overline{\xi}|}{1-|z|^2}=\frac{1}{2\cos \varphi -|1-z\overline{\xi}|}.$$ Thus,
$$\frac{1}{2\cos \varphi}\leq\frac{|1-z\overline{\xi}|}{1-|z|^2} \ \ \ \textrm{    for } z \in U, $$
and
$$\frac{|1-z\overline{\xi}|}{1-|z|^2}\leq \frac{2}{\cos \varphi}\ \ \ \textrm{    for } z \in U\cap d_1.$$

We say that a holomorphic function function $f$ in  $U$ [notation $f \in \OO(U)$]   has the non-tangential limit $L$ at $\xi \in \partial U$ if $f(z) \to L$ as $z \to \xi,$ $z \in \Gamma_\alpha(\xi);$ has radial limit $L$ at $\xi$ if $\lim_{t \to 1}f(t\xi)=L.$

Define the spherical derivative of $f(z)$ to be
$$f^\sharp(z)=\frac{|f'(z)|}{1+|f(z)|^2}. $$
Now we can reformulate Theorem 5 in \cite{[LV]} as follows:
\begin{theorem} \label{thm1}  If  $f \in \OO(U)$   has a radial limit
 at the point $\xi \in \partial U,$ then it has an non-tangential limit at this point if and only
if for any fixed $\alpha > 1$ in the non-tangential region $\Gamma_\alpha(\xi)$
\begin{equation}\label{on}
f^\sharp(z)\leq O\Big(\frac{1}{1-|z|}\Big).
\end{equation}

\end{theorem}

Let $B^n=\{z \in \CC^n : |z|<1 \}$ be a unit ball in $\CC^n,$ $n\geq 1.$ Consider the set $D_\alpha(\xi)\subset B^n$ such that
$$|1-(z,\xi)|<\frac{\alpha}{2}(1-|z|^2),$$
where $(z,\xi)=z_1\overline{\xi}_1+z_2\overline{\xi}_n$ and $|z|^2=(z,z).$

 Following Koranyi \cite{[Kor]}, we say that a holomorphic function $f$ in $B^n$ (henceforth, in symbols, $f\in \OO(B^n)$) has admissible limit $L$ at $\xi$ if for every
$\alpha>1$ for every sequence $\{z^j\}$ in $D_\alpha(\xi)$ that converges to $\xi,$ $f(z^j)\to L$ as $j\to \infty.$
(The case $L=\infty$ is not excluded.)

It is clear that the notions of admissible limit and non-tangential limit coincides when $n=1.$

The real tangent space to  $\partial B^n$ at point $\xi$ contains the complex tangent space $T^c_\xi(\partial B^n)$
and $\CC^n$ can be splitting $\CC^n=N_\xi(\partial B^n)\oplus T^c_\xi(\partial B^n). $ The complex line $N_\xi(\partial B^n)$ is called the complex normal
to  $\partial B^n$ at point $\xi.$

 For each $z$ near $\partial B^n$ denote by  $\zeta(z)$  the point on $\partial B^n$ closest to $z.$ Choose the coordinate system $\widetilde{z}_1, \ldots, \widetilde{z}_n$ in $\CC^n$
such that $\zeta(z)=0,$
$T^c_0(\partial B^2)=\{ (0, \widetilde{z}_1, \ldots,  \widetilde{z}_n)\},$ and  $N^c_0=\{(\widetilde{z}_1, 0,\ldots, 0 )\} $
and $\nu_0=(i,0,\ldots,0)$ is the inner normal to $\partial B^n$ at $\zeta(z).$ Set $'\widetilde{z}=(\widetilde{z}_2, \ldots, \widetilde{z}_n).$ Then $B^n=\{\widetilde{z}\in \CC^n : |\widetilde{z}_1-i|^2+|'\widetilde{z}_2|^2<1\}.$

Let polydisc $P_c(z)$ is defined to be the set of all $\widetilde{z} \in \CC^n$ whose coordinates $\widetilde{z}_1,\ldots , \widetilde{z}_n$ in $\CC^n$ satisfy the inequalities
$|\widetilde{z}_1-|z||<c(1-|z|), $ $|\widetilde{z}_\mu|<c\sqrt{1-|z|}, $ $\mu=2,\ldots, n,$ where $c<1/\sqrt{2}.$
For every $\widetilde{z} \in P_c(z)$ we have
$|\widetilde{z}_1-i|^2+|'\widetilde{z}_2|^2\leq 2|\widetilde{z}_1-|z||^2+2||z|-1|^2+|'\widetilde{z}_2|^2<2c^2(1-|z|)^2+[(n-1)c^2+4](1-|z|)<[4+(n+1)c^2](1-|z|)<1$ for all $z$ sufficiently close to $\partial B.$  It follows $P_c(z)\subset B$ for all $z$ sufficiently close to $\partial B.$
The one variable Cauchy's estimate shows that
\begin{equation}\label{Cest}
{\Big|\frac{\partial f}{\partial {\widetilde{z}}_1}(z)\Big|}\leq \frac{\sup_{\{w\in P(z)\}}|f(w)|}{c(1-|z|)}, \atop
{\Big|\frac{\partial f}{\partial {\widetilde{z}}_2}(z)\Big|}\leq \frac{\sup_{\{w\in P(z)\}}|f(w)|}{c\sqrt{1-|z|}}.
\end{equation}

This shows that in several variables the complex normal and complex tangential directions are not equivalent, therefore we  will distinct
the spherical derivative  of $f$ in point $z$ in the complex normal direction
( $={\Big|\frac{\partial f}{\partial {\widetilde{z}}_1}(z)\Big|}/({1+|f(z)|^2})$) and the complex tangential directions
( $={\Big|\frac{\partial f}{\partial {\widetilde{z}}_\mu}(z)\Big|}/({1+|f(z)|^2}), \mu=2,\ldots, n).$

Theorem \ref{thm1} fail to be  true in several variables. Look at the function  $f(z_1, z_2) =\frac{{z^2}_2}{1-z_1}.$  It is  holomorphic and bounded in $B^2,$ since $|f(z)|<(1-|z_1|)^2/(1-|z_1|)\leq 2.$
From (\ref{Cest}) follows that spherical derivative of $f$ in the complex normal  and complex tangential direction
grows no faster than $2c/(1-|z|)$ and  $2c/\sqrt{1-|z|}$ respectively. But this is not sufficient in order that the
existence of a limit along the normal for the function $f$ should imply the existence of an
admissible limit.

Indeed, put $z^j = (1 - 1/j, 1/\sqrt{j})$ for $j = 4, 5, \ldots. $ It is clear that $z^j \to \zeta=(1,0)$ as $j \to \infty.$ A simple calculation shows that $z^j \subset \mathcal  D_\alpha (\zeta)$ if $j$ is sufficiently large. Notice that
$\lim_{r\to 1-} f(r\zeta ) = \lim_{r\to 1-}0 = 0$ and
$f(z^j) = \frac{1/j}{1/j}= 1,$ and so $f$ does not have admissible limit at $\zeta.$

However,  for  $n  =  1$  in  the  estimate  (\ref{on})  (if  a  limit  along  the  normal  exists)  we  can  replace  the  right-hand  side  by  $o(1).$  It  turns  out  that  this  refined  estimate  solves  the  problem  for $ n  >  1.$

It was proved in \cite{[D2]} that if $f\in \OO(B^2),$  the spherical derivatives of $f$ in normal direction increases like $o(1/(1-|z|))$ and spherical derivatives of $f$ in complex tangential direction  increases like $o(1/\sqrt{1-|z|})$ then the
existence of a limit along the normal for the function $f$ should imply the existence of an
admissible limit.

 The main result of the article is the analogous result
for arbitrary domains with $C^2$-smooth boundary in $\CC^n,$ $n > 1.$

Montel \cite{[PM]} used normal  families in  a  simple but  ingenious way to investigate boundary   behavior   of   holomorphic   functions  in angular domains. We apply his method to investigate boundary   behavior   of   holomorphic   functions of several complex variables in admissible domains.

\section{A criterion of existence of admissible limits}
If
$ D$ is a bounded domain
in $\CC^n,$ $n >1,$ with $C^2$-smooth boundary $\partial D,$ then at each
$\xi \in \partial D$  the tangent space $T^c_\xi(\partial D)$ and the unit outward normal vector
$\nu_\xi$ are well-defined. We denote by $ T^c_\xi(\partial D)$ and $N^c_\xi(\partial D)$ the complex
tangent space and the complex normal space, respectively. The complex
tangent space at $\xi$ is defined as the $(n - 1)$ dimensional complex
subspace of $T_\xi(\partial D)$ and given by $ T^c_\xi(\partial D)=\{z \in \CC^n : (z,w)=0, \forall w \in  N^c_\xi(\partial D)\},$ where $(\cdot , \cdot)$ denotes canonical Hermitian product of $\CC^n.$ Let $\delta( z)$ denotes the Euclidean distance of  $z$ from $\partial D $ and $p(z,T_\xi(\partial D))$ is the Euclidean distance from  $z$ to the real tangent plane $T_\xi(\partial D).$

An admissible approach domain $\mathcal
A_{\alpha}(\xi) $ with vertex $\xi \in \partial D$ and aperture $\alpha > 0$ is defined as follows \cite{[sT]}:
\begin{equation}\label{01}
\mathcal
A_{\alpha}(\xi) = \{ \, z \in D \ : |(z-\xi,\nu_\xi)|<(1+\alpha)\delta_\xi(z) ,
|z-\xi|^2<\alpha\delta_\xi(z) \, \},\atop
{\delta_\xi(z)=\min \{\delta(z, \partial D), p(z,T_\xi(\partial D)) \}}.
\end{equation}

It is well known that the introduction of $\delta_\xi(z)$ and the second condition in (\ref{01}), i.e. $|z-\xi|^2<\alpha\delta_\xi(z)$
 only serves to rule out the pathological case when $\partial D$
has flat or concave points. For a ball $B^n=\{z\in \CC^n:|z|<1\}$ the set
$D_\alpha (\xi)$
essentially coincides with (\ref{01}).

\begin{definition}
The function $f,$ defined in a domain $D$ in $\CC^n$ has a limit $L,$ $L \in \overline{\CC},$ along the normal $\nu_\xi$ to $\partial D$  at the point $\xi$ iff $\lim_{t \to 0}f(\xi-t\nu_\xi)=L;$
$f$ has an \emph{admissible limit} $L,$  at $\xi \in \partial D$ iff
$$\lim_{\mathcal A_{\alpha}(\xi) \ni z \to \xi} f(z)= L$$ for every $\alpha >0;$
$f$ is admissible bounded at $\xi$ if $\sup_{z \in \mathcal A_{\alpha}(\xi)}|f(z)|<\infty$ for every $\alpha >0.$
\end{definition}

Let $x_j, y_j $ be the real coordinates of  $z \in \CC^n$ such that $z_j=x_j+iy_j.$ At times it will be convenient to use real variable notations by identifying $z$ with $(x_1, \zeta) \in  {\mathbb{R}^{2n}},$ where $\zeta=(y_1, x_2, y_2, \ldots , x_n, y_n) \in  {\mathbb{R}^{2n-1}}.$
After a unitary transformation of $\CC^n,$ if necessary,  we may assume the inner normal to $\partial D$ at $0$ points the positive $x_1$ direction,   $ T^c_0(\partial D)=\{ z \in \CC^n :  z_1 =0 \}.$
Let $ \pi: \CC^n \to N_0$ be an orthogonal projection, i.e., if $z=(z_1,  \ldots, z_n)$ then  $\pi(z)=(z_1, 0, \ldots, 0).$

Without loss of generality, there is a real valued $C^2$ function $\psi$ defined on $T_0(\partial D)=\{ (0,\zeta), \zeta \in  {\mathbb{R}^{2n-1}}\}$ so that $\partial D =\{ (\psi(\zeta),\zeta), \zeta \in  {\mathbb{R}^{2n-1}}\}$ and $D=\{ (x_1, \zeta), x_1 >\psi(\zeta)\}.$ (This is certainly true in the neighborhood of $0$ by the implicit function theorem, and our concerns are purely local here.) The fact that $T_0(\partial D)$ is tangent to $\partial D$ at $0$ implies $\nabla \phi(0)=0.$

For $z=(x_1, \zeta) \in D$ we set
$$d(z)=\min \{x_1,x_1-\psi(\zeta)\},$$
and define an approach region
\begin{equation}\label{adr}
A_{\alpha}(\xi) = \{ \, z \in D  : |z|^2<\alpha d(z), |y_1|<\alpha x_1 \}.
\end{equation}

The regions $A_{\alpha}(\xi)$ are "equivalent" to the admissible approach regions  (see \cite[Lemma 5.2]{[WR]}) in the sense that
$$\mathcal A_{\beta(\alpha)}(\xi)\subseteq A_{\alpha}(\xi)\subseteq \mathcal A_{\gamma(\alpha)}(\xi).$$

Set
$$({\bigtriangledown} F)^2= d^2(z)|{\bigtriangledown}_1 F(z)|^2+d(z)|{\bigtriangledown}_{2,n} F(z)|^2,$$
where
$$|{\bigtriangledown}_1 F(z)|^2=\Big|\frac{\partial F}{\partial {z}_1}(z)\Big|^2, \ \ \
|\bigtriangledown_{2,n} F(z)|^2=\sum_{j=2}^n\Big|\frac{\partial F}{\partial {z}_j}(z) \Big|^2.$$

We begin with proposition.

\begin{proposition} \label{prop} Let $D$ be a  domain in $\CC^n,$ $ n>1,$ with $C^2$-smooth boundary. Suppose that the function $f\in \OO(D)$ has a limit $L$ along the normal $\nu_\xi$ to $\partial D$  at the point $\xi$
equal to $L,$ $ L\neq \infty.$ If $$(1+|f(z)|^2)^{-1}{\nabla} f(z)$$ is admissible bounded at $\xi,$ then $f$ admissible bounded at $\xi.$
\end{proposition}
\begin{proof}  Assume $\xi=0.$  Since the domain $D$ has $C^2$-smooth boundary, then there is a constant $r>0$ such that the ball $B_r(-r\nu_0)\subset D$ and $\partial B_r(-r\nu_0)\cap \partial D=\{0\}.$

Let the function $f$ has the finite  limit $L$ along the normal $\nu_0$ to $\partial D$  at the point $0.$
Since $d(z)\geq |r-z_1|\geq r-|z_1|\geq \frac{1}{2r}(r^2-|z_1|^2)$ for all $z\in B_r(-r\nu_0)$ sufficiency close to $0$  we have $$(r^2-|z_1|^2)\frac{|\frac{\partial f}{\partial {z}_1}(\pi(z))|}{1+|f(\pi(z))|^2}<\frac{{\nabla} f(\pi(z))}{1+|f(\pi(z))|^2}<O(1),  \ \ \  z \in   A_\alpha(0)\cap N_0^c(\partial D).$$
Therefore $f(\pi(z))$ fulfills all the hypotheses of Theorem \ref{thm1}. Hence $f(\pi(z))\to L $ as $z\to 0,$ $z \in   A_\alpha(0)\cap N_0^c(\partial D).$

Assume, to reach a contradiction,
 that $f$ is not admissible bounded at $0.$ Let $\{z^m\}$ be any sequence of points from $ A_\alpha(0)$ such that $z^m\to 0$ as $m \to \infty$ and $f(z^m)\to \infty$ as $m \to \infty.$

For the biholomorphic mapping $\Phi_b(z)=(w_1(z), \ldots, w_n(z)),$ where $w_1(z)=\frac{z_1-b_1}{2c d(b)},$ $w_\mu(z)=\frac{z_\mu-b_\mu}{2c \sqrt{d(b)}},$ $\mu=2, \ldots, n,$ the polydisc $$P(b,c)=\{z\in \CC^n : |z_1-b_1|<cd(b), |w_\mu-b_\mu|<c \sqrt{d(b)}, \mu=2, \ldots, n,\}$$ is mapped to the unit polydisc $U^n=\{w\in \CC^n : |w_\mu|<1, \mu= 1,\ldots, n\}.$ By \cite[Lemma 7.2]{[WR]} the exists $c=c(\alpha)$ such that $P(b,c) \subset A_{2\alpha}(0)$ for all sufficiently small $b \in A_\alpha(0).$ Therefore with each point $b\in  A_\alpha(0)$ sufficiently close to $0$ we can associate a function $g_b=f(\Psi^{-1}_b(w))$ which is well defined and holomorphic in polydisc $U^n.$

By \cite[Lemma 5.2]{[WR]} there exists $c=c(\alpha)$ so that if $z=(x_1,\zeta)$ sufficiently small and $|z|<\alpha d(z)$ we have $d(z)\geq cx_1.$
Let $t$ be an arbitrary point of the interval $[z^m, \pi(z^m)].$ Note that $x^m_1\geq d(t)\geq cx^m_1.$

Choose an integer $N$ such that  $\alpha < cN/2.$ From the definitions of the set
$ A_\alpha(0)$ it follows that $|z^m-\pi(z^m)|^2<cN x^m_1/2.$
Then any interval $[z^m, \pi(z^m)]$ may be covered by $k_m$ polydiscs, where $k_m< N+1,$
$$P_{m,k}(c)=P(b^{m,k},c)=\atop \{z\in \CC^n :|z_1-b^{m,k}_1|<cd(b^{m,k}), |z_\mu-b^{m,k}_\mu|<c \sqrt{d(b^{m,k})}, \mu= 2 ,\ldots, n\}$$
such that $b^{m,1}=(z^m_1, '0),$ $b^{m,k_m}=z^m,$ $ b^{m,k} \in [z^m, \pi(z^m)],$ $k=2, \ldots, k_m-1,$ $P_{m,k}(c/2)\ni b^{m,k+1}$ (and hence $P_{m,k}(c/2)\cap P_{m,k+1}(c/2)\neq \emptyset$) for all $m\geq 1,$ $k<k_m.$ To each point $ b^{m,k}$ we associate a function $g_{m,k}=g_{b^{m,k}}$ as above.

Set $G^m=g_{m,k_m},$ $ m\geq 1 .$ Since $f(z^m)\to \infty$ as $m\to \infty$  and $P_{m,k_m}(c)\ni z^m $ we have $g_{m,k_m}(0)=f(z^m)\to \infty$ as $m \to \infty.$ Suppose that there is a sequence of points $\{w^m\}$ which belongs to some  polydisc $P_2,$
$\overline{P}_2\subset U^n,$ such that $G^m(w^m)\not \to \infty$ as $m \to \infty.$ It follows that the family $\{G^m\}$ is not normal in $U^n$ and by Marty's criterion (see, e.g., \cite{[D1]}) there are points $p^m \in \overline{P}_2$ and vectors $v^m \in \CC^n$ with $|v^m|=1$  such that
\begin{equation}\label{e21}
\frac{(dG^m_{{p}^m}(v^m),dG^m_{{p}^m} (v^m))}{(1+|G^m({p}^m)|^2)^2}>m,
\ \ \ (m=1,2, \ldots ),
\end{equation}
where
$$dG^m_{{p}^m} (v^m)=\sum_{\mu=1}^n\frac{\partial G^m}{\partial w_\mu }(p^m)v^m_\mu.$$
According to the rule of differentiation of composite functions
$$ \frac{\partial G^m}{\partial w_1}(p^m)=c d(b^{m,1}) \frac{\partial f}{\partial z_1}(t^m)
\atop \frac{\partial G^m}{\partial w_\mu}(p^m)=c \sqrt{d(b^{m,1})} \frac{\partial f}{\partial z_\mu}(t^m), ( \mu=2,\ldots, n),
 $$
where $t^m=\Psi^{-1}_{b^{m,1}}(p^m)\in P_{m,1}(c)\subset  A_{2\alpha}(0). $ By \cite[Lemma 5.2]{[WR]}
there exists $c_1=\min \{1/2, 1/2K\alpha \}$ so that if $z=(x_1,\zeta) \in A_{2\alpha}(0)$ is sufficiently small then
$x_1>d(z)\geq c_1x_1.$ Since $b^{m,1}_1=x^m_1$ and $(1-c)x^m_1\leq Re \, t^m_1\leq (1+c)x^m_1$ we have
$$ \frac{c_1}{1+c}\leq\frac{d(b^{m,1})}{ d(t^m)}\leq \frac{1}{c_1(1-c)}.$$
This, together with  the Bunyakovski\u{i}-Schwarz inequality, implies from (\ref{e21}) that
$$O(1)\frac{({\nabla} f(t^m))^2}{(1+|f(t^m)|^2)^2}>m.$$
It follows that $\frac{({\nabla} f(z))^2}{(1+|f(z)|^2)^2}$ is not admissible bounded in $0,$ a contradiction with hypothesis of the theorem. Therefore sequence $\{G^m\}$
converges  uniformly on compact subsets of $U^n$ to $\infty.$ Put now $G^m =g_{m,\{k_m-1\}},$ $m\geq 1.$ (Note that we set $g_{m,\{k_m-1\}}\equiv g_{m,k_m}$ if $k_m-1\leq 0.$) Since $P_{m,k_m-1}(c/2)\cap P_{m,k_m}(c/2)\neq \emptyset$ we  have $G^m(0) \to \infty$ as $m \to \infty$ and we may repeat the above
argument.  After
finite number of steps the proof will be completed since $P_{m,1}(c)\ni \pi{(z^m)}$ and $f(\pi{(z^m)})\to L$ as $m \to \infty.$ We get
${({\nabla} f(z))^2}/{(1+|f(z)|^2)^2}$ is not admissible bounded in $0,$ contrary to the hypothesis on ${{\nabla} f(z)}/{(1+|f(z)|^2)}.$
This contradiction proves our claim.
\end{proof}

\begin{theorem} \label{t1} Let $D$ be a  domain in $\CC^n,$ $ n>1,$ with $C^2$-smooth boundary. If a function $f$ holomorphic in $D$
has a limit along the normal $\nu_\xi$ at a point $\xi \in \partial D,$ then it has an admissible limit at this point
if and only if for every $\alpha> 0$
\begin{equation}\label{e1}
(1+|f(z)|^2)^{-1}{\nabla} f(z)\to 0
\end{equation}
as $z\to \xi,$ $ z\in  A_\alpha(\xi).$
\end{theorem}

\begin{proof} \emph{Necessity.}  Assume $\xi=0,$ without loss of generality. First, let $f$ has finite admissible limit $L$ at $0.$  Without loss of generality, assume   $L=0$ at $0.$
Let $P_1(z)$ denote the polydisc centered
at $z,$ whose radii are essentially $c x_1,$ $c\sqrt{x_1},$ $\ldots,$ $c\sqrt{x_1},$
with $c$ sufficiently small. By \cite[Lemma 7.2] {[WR]} exists $c=c(\alpha)$ such that  $P_1(z) \subset A_{2\alpha }(\xi).$ Let $P(z)$ denote the polydisc centered
at $z,$ whose radii are essentially $c d(z),$ $c\sqrt{d(z)},$ $\ldots,$ $c\sqrt{d(z)}.$ Since $d(z)=\min\{x_1, x_1-\psi(\zeta)\}\leq x_1$ we have $P(z)\subseteq P_1(z)\subset D.$ The one variable Cauchy's estimate shows that
$${|\bigtriangledown}_1 f(z)|\leq \frac{\sup_{\{w\in P(z)\}}|f(w)|}{cd(z)}, \atop
{|\bigtriangledown}_{2,n} f(z)|\leq \frac{\sup_{\{w\in P(z)\}}|f(w)|}{c\sqrt{d(z)}}.$$
Since $f(z)\to 0$ as $z \to 0,$ $z \in  A_\alpha(0),$ we have
$${\nabla} f(z)\to 0$$ as $z \to 0,$ $z \in  A_\alpha(0).$
It remains to observe that  $\nabla f(z)\geq(1+|f(z)|^2)^{-1}\nabla f(z).$

If the function $f$ has an admissible limit at the point $0$ equal to infinity, then for any
$\alpha > 0$ there is a $\varepsilon > 0$ such that $1/f \in \OO( A_\alpha(0))\cap B_\varepsilon(0)).$ The function $F=1/f$ has an admissible
limit equal to zero at the point $0,$ so, as we have proved, $F$ satisfies (\ref{e1}). It remains to  observe that outside the zeros of $f$ we obviously have
$(1+|F(z)|^2)^{-1}\nabla  F(z)=(1+|f(z)|^2)^{-1}\nabla  f(z).$

\emph{Sufficiency.}
\textbf{(a)} Suppose that the function f has a limit $L$ along the normal $\nu_0$ to $\partial D$  at the point $0$
equal to $L,$ $ L\neq \infty.$

 We may assume, without loss of generality,
that $L=0.$
Write
$$f(z) = \{f(z) - f(z_1,0,\ldots, 0)\} + f(z_1, 0,\ldots, 0).$$  The first term on
the right side is dominated by
$|z(1) - z(0)| \sup_{\{0 < t < 1\}} {|{\bigtriangledown}_{2,n}f(z(t))|}, $ where $z(t) = (z_1,z_2 t ,\ldots,z_n t ),$ $t \in [0,1],$
If $z \in  A_\alpha(0),$ then by \cite[Lemma 7.3]{[WR]} $z(t) \in  A_\alpha(0),$ $t \in [0,1],$ and there $d(z(t)) \approx d(z)$
while $|z(1) - z(0)| < \alpha\sqrt{ d(z)}.$ (The expression
$A\approx B$ means that there are positive constants $c_1$ and $c_1$ such that $c_1A < B < c_2A.$) By Proposition \ref{prop} $f$ is admissible bounded in $0$ and therefore
$$|z(1) - z(0)| \sup_{\{0 < t < 1\}} |\bigtriangledown_{2,n}f(z(t))| \leq O(1)\frac{{\bigtriangledown}f(z(t_0))}{1+|f(z(t_0)|^2},$$
 where $0\leq t_0\leq 1.$
Since  ${\bigtriangledown}f(z(t_0))/(1+|f(z(t_0))|)\to 0$ as $z(t_0) \to 0 $
we have that $f(z) - f(z_1,0,\ldots, 0)\to 0$ as $z \to 0$ in $ A_\alpha(0).$ Since $f(z_1, 0,\ldots, 0)\to 0$ as $z \to 0$
in $ A_\alpha(0)$ we conclude that
$$\lim_{ A_\alpha(0) \ni z \to 0} f(z)=0.$$

 The above proof is quite analogous to the proof in  \cite[p, 68]{[sT]}.

\textbf{(b)} Let the function $f$ has the infinite  limit  along the normal $\nu_0$ to $\partial D$  at the point $0.$ Let $\{z^m\}$ be any sequence of points from $ A_\alpha(0)$ such that $z^m\to 0$ as $m \to \infty.$
As in the proof of Proposition \ref{prop} let $\{G^m\},$ be a sequence of function defined on $U^n.$  Then as in Proposition \ref{prop} we obtain $f(z^m)\to\infty$ as $m \to \infty.$
Since the
sequence of points $\{z^m\}$ was arbitrary, by definition this means that  $f$ has the admissible limit equal to infinity at the point $0.$ The theorem is proved. \end{proof}

 For each $z$ near $\partial D$ denote by  $\zeta(z)$  the point on $\partial D$ closest to $z.$ Choose the coordinate system $\widetilde{z}_1,\ldots, \widetilde{z}_n$
such that $\zeta(z)=0,$ and
$\{\widetilde{z}\in \CC^n : (\widetilde{z}_1, 0 \ldots, 0)\} = N^c_0 (\partial D),$ and $\{\widetilde{z}\in \CC^n : (0,\widetilde{z}_2,\ldots, \widetilde{z}_n)\} = T^c_0(\partial D),$
and $\nu_0=(1, 0, \ldots, 0).$
Denote by $grad_{\CC} F=\Big(\frac{\partial F}{\partial \widetilde{z}_1}, \ldots, \frac{\partial F}{\partial \widetilde{z}_n} \Big)$ the  complex gradient of function $F.$
Write also
$$|\widetilde{\bigtriangledown}_1 F|^2=\Big|\frac{\partial F}{\partial \widetilde{z}_1}\Big|^2, \atop
|\widetilde{\bigtriangledown}_{2,n} F|^2=\sum_{j=2}^n\Big|\frac{\partial F}{\partial \widetilde{z}_j}\Big|^2.$$
Then $ |grad_{\CC}F|^2=|\widetilde{\bigtriangledown}_1 F|^2+|\widetilde{\bigtriangledown}_{2,n} F|^2$
but this splitting varies (with the decomposition $\CC^n=N_{\zeta(z)}\oplus T^c_{\zeta(z)}$ ) as
$z$ varies in $A_\alpha(\xi).$

We need to observe that (the proof is the same as in \cite[pp. 61-62]{[sT]})
\begin{equation}\label{inst}
d^2(z)|\bigtriangledown_1 F|^2+d (z)|\bigtriangledown_{2,n} F|^2 \approx d^2(z)|\widetilde{\bigtriangledown}_1 F|^2+d (z)|\widetilde{\bigtriangledown}_{2,n} F|^2 \ \
\ \ (z \in A_\alpha(\xi)).
\end{equation}
We write $A\approx B$
if the ration $|A|/|B|$ is bounded between two positive constants.

We call $$\frac{\Big|\frac{\partial f}{\partial \widetilde{z}_1}(z)\Big|}{1+|f(z)|^2} \textrm{ and }
\frac{\Big|\frac{\partial f}{\partial \widetilde{z}_\mu}(z)\Big|}{1+|f(z)|^2} \ \  (\mu=2, \ldots, n)$$
the spherical derivative of $f(z)$ in the normal and  complex tangent direction, respectively.
From (\ref{inst}) follows that Theorem \ref{t1} is actually equivalent to:
\begin{theorem} \label{teor2} Let $D$ be a  domain in $\CC^n,$ $ n>1,$ with $C^2$-smooth boundary.  If a holomorphic function $f$ has a limit along
the normal to $\partial D$ at the point $\xi,$ then at the point $\xi\in \partial D$ the function $f$ has an admissible
limit if and only if in every admissible domain with vertex $\xi$ the spherical derivative of $f$ in the normal and complex tangent directions
increases like $o(1/d(z))$ and $o(1/\sqrt{d(z)}),$ respectively.
\end{theorem}

The  example in the beginning of this article shows that  the Lindel\"{o}f principle for bounded functions  -- formulated in terms of admissible convergence -- fails. However the following refinement of Lindel\"{o}f's theorem holds.

\begin{theorem}\label{t3} Let $D$ be a  domain in $\CC^n,$ $ n>1,$ with $C^2$-smooth boundary. If a function $f$  in $D$
has a limit $L,$ $L\in \overline{\CC},$ along the normal $\nu_\xi$ at a point $\xi \in \partial D,$  and in every admissible domain with vertex $\xi$ the function $f$ is holomorphic, $L$ is his omitted value  and the spherical derivative of $f$ in the normal and complex tangent directions
grows no faster than $K/d(z)$ and $K/\sqrt{d(z)},$ respectively, then $f$ has an admissible limit $L$ at
 $\xi.$
\end{theorem}
\begin{proof} By hypothesis of the theorem $L \not \in f(D)$ then $(f(z)-L)^{-1}$ is holomorphic on $D$ and has a radial limit at $\xi$ equal to $\infty.$
It is thus sufficient to consider the case
 $L=\infty.$

 By Theorem \ref{thm1} and hypothesis on $f$ we have $f(\pi(z))\to \infty $ as $z \to \xi,$ $z \in A_\alpha(\xi)\cap N_\xi^c(\partial D).$ Let $\{z^m\}$ be any sequence of points from $ A_\alpha(\xi)$ such that $z^m\to \xi$ as $m \to \infty.$
 Since the spherical derivative of $f$ in the normal and complex tangent directions
grows no faster than $K/d(z)$ and $K/\sqrt{d(z)},$ respectively, from (\ref{inst}) follows
$$d^2(z)|\bigtriangledown_1 F(z)|^2+d (z)|\bigtriangledown_{2,n} F(z)|^2 \leq O(1) \ \
\ \ (z \in A_\alpha(\xi)).$$
Using the notation introduced in the proof of Proposition \ref{prop}, the Bunyakovski\u{i}-Schwarz inequality and the fact that $d(b^{m,1})\approx d(z)$ for all $z\in P_{m,1}$ it follows that
$$\frac{(dG^m_{{p}}(v),dG^m_{{p}} (v))}{(1+|G^m({p})|^2)^2}\leq O(1)
\ \ \ (m=1,2, \ldots )$$
for all $p\in P$ and all $v\in \CC^n,$ $|v|=1.$

By Marty's criterion (see, e.g., \cite{[D1]}) the family $\{G^m\}$ are normal in $U^n.$
Since $G^m(\pi(z^m)=g_{m,1}(0)\to \infty$ as $m \to \infty$ it follows that the  sequence $\{G^m\}$
converges  uniformly on compact subsets of $U^n$ to $\infty.$   Then as in Theorem \ref{t1} we obtain $f(z^m)\to\infty$ as $m \to \infty.$

Since the
sequence of points $\{z^m\}$ chosen from $ A_\beta(0)$ is arbitrary, this completes the proof that the
function $f$ has the admissible limit $L$ at the point $\xi.$ The theorem is proved.

\end{proof}

\begin{theorem}\label{t4} Let $D$ be a  domain in $\CC^n,$ $ n>1,$ with $C^2$-smooth boundary. Let in every admissible domain with vertex $\xi$ the function $f$ is holomorphic and its  spherical derivative  in the normal and complex tangent directions
grows no faster than $K/d(z)$ and $K/\sqrt{d(z)},$ respectively. If
$$\lim _{A_\beta(\xi)\ni z\to \xi}f(z) =L \textrm{  for some } \beta >0,$$
then $f$ has an admissible limit at
 $\xi.$
\end{theorem}
\begin{proof} Fix $\alpha > \beta.$ Let $\{z^m\}$ be an arbitrary sequence of $A_\alpha(\xi).$ Let $G^m=g_{m,1},$ $ m\geq 1 ,$
be the sequence of function defined as in proof of Proposition \ref{prop}. The family $\{g_{m,1}\}$ is normal on $P$ (this was proved in Theorem  \ref{teor2}). Since $f(z)\to L$ as $z\to 0$ in $ A_\beta(0),$ without lost a generality,
we may assume that $P_{m,1}(c) \subset A_\beta(0)$ for all $m=1, 2, \ldots \, .$ Hence $G^m$ tends to  $L$ uniformly on every compact subset of $P.$

By \cite[Lemma 5.2]{[WR]}
there exists $c_1=\min \{1/2, 1/2K\alpha \}<1/2$ so that if $z=(x_1,\zeta) \in A_{2\alpha}(0)$ is sufficiently small then
$x_1>d(z)\geq c_1x_1.$ Since $b^{m,1}_1=b^{m,2}_1=x^m_1$  we have
$$ c_1\leq\frac{d(b^{m,2})}{d(b^{m,1})}\leq \frac{1}{c_1}.$$

Since
$$\Psi^{-1}_{b^{m,2}}(w)=(cd(b^{m,2})w+b^{m,2}_1, c\sqrt{d(b^{m,2})}w+b^{m,2}_2, \ldots, c\sqrt{d(b^{m,2})}w+b^{m,2}_n),$$
$|b^{m,2}_1-b^{m,1}_1|<c/2\cdot d(b^{m,1}),$ and $|b^{m,2}_\mu-b^{m,1}_\mu|<c/2\sqrt{d(b^{m,1})},$ $\mu=1,2,\ldots, n,$ the little calculation shows that for for all $w \in P(0,c_1/4)\subset P$
$$|w_1cd(b^{m,2})-b^{m,2}_1|<\frac{cc_1}{4}\frac{d(b^{m,2})}{d(b^{m,1})}d(b^{m,1})+\frac{c}{2}d(b^{m,1})
<\frac{3c}{4}d(b^{m,1})$$ and
$$|w_\mu cd(b^{m,2})-b^{m,2}_\mu|<(\frac{cc_1}{4\sqrt{c_1}}+\frac{c}{2})\sqrt{d(b^{m,1})}
<\frac{3c}{4}d(b^{m,1})\ \ \ \mu=1,2,\ldots, n,$$
 It follows $g_{m,2}$ takes the same values on $ P(0,c_1/4)$ as $f$ on $\Psi^{-1}_{b^{m,2}}(P(0,c_1/4))\subset P_{m,1}(c)$ hence $g_{m,2}\to L$  on $ \overline{P(0,c_1/5)}\subset P.$

The family $\{g_{m,2}\}$ is normal on $P$ (this was proved in Theorem  \ref{teor2}) hence
the family $\{g_{m,2}\}$ also tends to  $L$ uniformly on compact subsets of $P.$ After finite steps we obtain that $f(z^m)\to L$ as $m \to \infty.$ Since the
sequence of points $\{z^m\}$ chosen from $ A_\beta(0)$ is arbitrary, this completes the proof that the
function $f$ has the admissible limit $L$ at the point $\xi.$ The theorem is proved.
\end{proof}

For bounded holomorphic functions this theorem appears in Chirka's paper \cite{[C]}, with the proof sketched there relying on certain estimates on harmonic measures. A proof based on a different method was given by Ramey \cite[Theorem 2]{[WR]}.


\end{document}